\documentclass[11pt]{article}
\usepackage[margin=1in]{geometry} 
\usepackage{amssymb,amsfonts,amsmath,bbm,mathrsfs,stmaryrd,mathtools}
\usepackage{xcolor}
\usepackage{url}

\usepackage[shortlabels]{enumitem}
\usepackage{tensor}

\usepackage[colorlinks,
linkcolor=black!75!red,
citecolor=blue,
pdftitle={},
pdfproducer={pdfLaTeX},
pdfpagemode=None,
bookmarksopen=true,
bookmarksnumbered=true,
backref=page]{hyperref}

\usepackage{tikz}
\usetikzlibrary{arrows,calc,decorations.pathreplacing,decorations.markings,intersections,shapes.geometric,through,fit,shapes.symbols,positioning,decorations.pathmorphing}

\usepackage{braket}

\usepackage[amsmath,thmmarks,hyperref]{ntheorem}
\usepackage{cleveref}

\creflabelformat{enumi}{#2#1#3}

\crefname{section}{Section}{Sections}
\crefformat{section}{#2Section~#1#3} 
\Crefformat{section}{#2Section~#1#3} 

\crefname{subsection}{\S}{\S\S}
\AtBeginDocument{%
  \crefformat{subsection}{#2\S#1#3}%
  \Crefformat{subsection}{#2\S#1#3}% 
}

\crefname{subsubsection}{\S}{\S\S}
\AtBeginDocument{%
  \crefformat{subsubsection}{#2\S#1#3}%
  \Crefformat{subsubsection}{#2\S#1#3}% 
}

\theoremstyle{plain}

\newtheorem{lemma}{Lemma}[section]
\newtheorem{proposition}[lemma]{Proposition}
\newtheorem{corollary}[lemma]{Corollary}
\newtheorem{theorem}[lemma]{Theorem}

\theoremstyle{plain}
\theoremnumbering{Alph}
\newtheorem{theoremN}{Theorem}

\newtheorem{corollaryN}[theoremN]{Corollary}

\theoremstyle{plain}
\theorembodyfont{\upshape}
\theoremsymbol{\ensuremath{\blacklozenge}}

\newtheorem{definition}[lemma]{Definition}
\newtheorem{example}[lemma]{Example}

\newtheorem{remark}[lemma]{Remark}
\newtheorem{remarks}[lemma]{Remarks}

\crefname{definition}{definition}{definitions}
\crefformat{definition}{#2definition~#1#3} 
\Crefformat{definition}{#2Definition~#1#3} 

\crefname{ex}{example}{examples}
\crefformat{example}{#2example~#1#3} 
\Crefformat{example}{#2Example~#1#3} 

\crefname{exs}{example}{examples}
\crefformat{examples}{#2example~#1#3} 
\Crefformat{examples}{#2Example~#1#3} 

\crefname{remark}{remark}{remarks}
\crefformat{remark}{#2remark~#1#3} 
\Crefformat{remark}{#2Remark~#1#3} 

\crefname{remarks}{remark}{remarks}
\crefformat{remarks}{#2remark~#1#3} 
\Crefformat{remarks}{#2Remark~#1#3} 

\crefname{convention}{convention}{conventions}
\crefformat{convention}{#2convention~#1#3} 
\Crefformat{convention}{#2Convention~#1#3} 

\crefname{notation}{notation}{notations}
\crefformat{notation}{#2notation~#1#3} 
\Crefformat{notation}{#2Notation~#1#3} 

\crefname{table}{table}{tables}
\crefformat{table}{#2table~#1#3} 
\Crefformat{table}{#2Table~#1#3}

\crefname{lemma}{lemma}{lemmas}
\crefformat{lemma}{#2lemma~#1#3} 
\Crefformat{lemma}{#2Lemma~#1#3} 

\crefname{proposition}{proposition}{propositions}
\crefformat{proposition}{#2proposition~#1#3} 
\Crefformat{proposition}{#2Proposition~#1#3} 

\crefname{propositionN}{proposition}{propositions}
\crefformat{propositionN}{#2proposition~#1#3} 
\Crefformat{propositionN}{#2Proposition~#1#3} 

\crefname{corollary}{corollary}{corollaries}
\crefformat{corollary}{#2corollary~#1#3} 
\Crefformat{corollary}{#2Corollary~#1#3} 

\crefname{corollaryN}{corollary}{corollaries}
\crefformat{corollaryN}{#2corollary~#1#3} 
\Crefformat{corollaryN}{#2Corollary~#1#3} 

\crefname{theorem}{theorem}{theorems}
\crefformat{theorem}{#2theorem~#1#3} 
\Crefformat{theorem}{#2Theorem~#1#3} 

\crefname{theoremN}{theorem}{theorems}
\crefformat{theoremN}{#2theorem~#1#3} 
\Crefformat{theoremN}{#2Theorem~#1#3} 

\crefname{enumi}{}{}
\crefformat{enumi}{#2#1#3}
\Crefformat{enumi}{#2#1#3}

\crefname{assumption}{assumption}{Assumptions}
\crefformat{assumption}{#2assumption~#1#3} 
\Crefformat{assumption}{#2Assumption~#1#3} 

\crefname{construction}{construction}{Constructions}
\crefformat{construction}{#2construction~#1#3} 
\Crefformat{construction}{#2Construction~#1#3} 

\crefname{equation}{}{}
\crefformat{equation}{(#2#1#3)} 
\Crefformat{equation}{(#2#1#3)}

\numberwithin{equation}{section}

\theoremstyle{nonumberplain}
\theoremsymbol{\ensuremath{\blacksquare}}

\newtheorem{proof}{Proof}

\newcommand\bC{{\mathbb C}}
\newcommand\bD{{\mathbb D}}
\newcommand\bE{{\mathbb E}}

\newcommand\bG{{\mathbb G}}
\newcommand\bH{{\mathbb H}}

\newcommand\bK{{\mathbb K}}

\newcommand\bP{{\mathbb P}}
\newcommand\bQ{{\mathbb Q}}
\newcommand\bR{{\mathbb R}}
\newcommand\bS{{\mathbb S}}
\newcommand\bT{{\mathbb T}}
\newcommand\bU{{\mathbb U}}

\newcommand\bZ{{\mathbb Z}}

\newcommand\cA{{\mathcal A}}

\newcommand\cC{{\mathcal C}}

\newcommand\cH{{\mathcal H}}
\newcommand\cI{{\mathcal I}}

\newcommand\cK{{\mathcal K}}
\newcommand\cL{{\mathcal L}}
\newcommand\cM{{\mathcal M}}

\newcommand\cO{{\mathcal O}}
\newcommand\cP{{\mathcal P}}

\DeclareMathOperator{\id}{id}
\DeclareMathOperator{\im}{im}
\DeclareMathOperator{\End}{\mathrm{End}}

\DeclareMathOperator{\Irr}{\mathrm{Irr}}

\DeclareMathOperator{\spn}{\mathrm{span}}

\newcommand{\cat}[1]{\textsc{#1}}

\newcommand\spr[1]{\cite[\href{https://stacks.math.columbia.edu/tag/#1}{Tag {#1}}]{stacks-project}}
\newcommand{\qedhere}{\mbox{}\hfill\ensuremath{\blacksquare}}

\newcommand{\xrightarrowdbl}[2][]{%
  \xrightarrow[#1]{#2}\mathrel{\mkern-14mu}\rightarrow
}

\title{(Quantum) discreteness, spectrum compactness and uniform continuity}
\author{Alexandru Chirvasitu}

\begin{document}

\date{}

\newcommand{\Addresses}{{% additional braces for segregating \footnotesize
  \bigskip
  \footnotesize

  \textsc{Department of Mathematics, University at Buffalo}
  \par\nopagebreak
  \textsc{Buffalo, NY 14260-2900, USA}  
  \par\nopagebreak
  \textit{E-mail address}: \texttt{achirvas@buffalo.edu}

  % % \medskip
  % % 
  % % \textsc{Department of Mathematics, INSTITUTION}
  % % \par\nopagebreak
  % % \textsc{ADDRESS}
  % % \par\nopagebreak
  % % \textit{E-mail address}: \texttt{??}
  % % 

}}

\maketitle

\begin{abstract}
  We prove a number of results linking properties of actions by compact groups (both quantum and classical) on Banach spaces, such as uniform continuity, spectrum finiteness and extensibility of the actions across several constructions. Examples include: (a) a unitary representation of a compact quantum group induces a continuous action on the $C^*$-algebra of bounded operators if and only if it has finitely many isotypic components, and hence is uniformly continuous; (b) a compact quantum group is finite if and only if its continuous actions on $C^*$-algebras lift to continuous actions on either the multiplier algebras or von Neumann envelopes thereof; (c) a (classical) locally compact group $\mathbb{G}$ is discrete if and only if the forgetful functor from $\mathbb{G}$-acted-upon compact $T_2$ spaces back to compact $T_2$ spaces creates coproducts; (d) a representation of a linearly reductive quantum group has finitely many isotypic components if and only if its restrictions to two topologically-generating quantum subgroups, one of which is normal, do; (e) equivalent characterizations of uniform continuity for actions of compact groups on Banach spaces, e.g. that such an action is uniformly continuous if and only if its restrictions to a pro-torus and to pro-$p$ subgroups are. 
\end{abstract}

\noindent {\em Key words: compact group; quantum group; representation; uniformly continuous; Banach space; compact operator; multiplier algebra; pro-torus; profinite; tensor product; Hopf algebra; cosemisimple; linearly reductive; normal quantum subgroup; exact sequence}

\vspace{.5cm}

\noindent{MSC 2020: 20G42; 22D12; 22D10; 22D05; 22E45; 46M05; 46L06; 16T15; 16T05; 43A40; 43A65}

\tableofcontents

%%%%%%%%%%%%%%%%%%%%%%%%%%%%%%%%
%%%%%%%%%%%%%%%%%%%%%%%%%%%%%%%%
\section*{Introduction}

This note gathers a number of observations revolving around the issue of extending compact (quantum) group actions across constructions such as compactification, completion of some sort, coproducts, etc.

To help make the matter concrete (and exemplify the problems at hand), consider the discussion preceding \cite[Definition 2.7.1]{zbMATH04028180}: given an action of a group $\bG$ (mostly assumed compact throughout \cite{zbMATH04028180}) on a $C^*$-algebra $A$, continuous in the usual (\cite[\S 7.4.1]{ped-aut}, \cite[Definition II.10.3.1]{blk}) sense that
\begin{equation*}
  \bG\ni g
  \xmapsto{\quad}
  ga\in A
\end{equation*}
is continuous in the {\it norm} topology of $A$ for every $a\in A$, there seems to be no reason why the corresponding action on the {\it multiplier algebra} \cite[\S II.7.3]{blk} $M(A)$ will again be continuous in the same sense. This in turn motivates \cite[Definition 3.1.1]{zbMATH04028180} introducing the subspace $M_{\bG}(A)\le M(A)$ where said continuity does obtain, etc.

Explicit examples of the flagged (potential) pathology are not give in \cite[\S 2.7]{zbMATH04028180}, but they are easy to produce and then turn around into characterizations of finiteness/discreteness for the acting compact groups or the spectra of the actions in question. Much of this, moreover, extends to compact {\it quantum} groups. Recall, for instance, \cite[Definition 1.1]{wor-cqg} or \cite[Definition 1.1.1]{NeTu13}:

\begin{definition}\label{def:cqg}
  A {\it compact quantum group} $\bG$ consists of
  \begin{itemize}
  \item a unital $C^*$-algebra $\cC(\bG)$
  \item equipped with a $C^*$-morphism $\cC(\bG)\xrightarrow{\Delta}\cC(\bG)\underline{\otimes} \cC(\bG)$, {\it coassociative} in the sense that
    \begin{equation*}
      \begin{tikzpicture}[auto,baseline=(current  bounding  box.center)]
        \path[anchor=base] 
        (0,0) node (l) {$\cC(\bG)$}
        +(3,.5) node (u) {$\cC(\bG)^{\underline{\otimes} 2}$}
        +(3,-.5) node (d) {$\cC(\bG)^{\underline{\otimes} 2}$}
        +(6,0) node (r) {$\cC(\bG)^{\underline{\otimes} 3}$}
        ;
        \draw[->] (l) to[bend left=6] node[pos=.5,auto] {$\scriptstyle \Delta$} (u);
        \draw[->] (u) to[bend left=6] node[pos=.5,auto] {$\scriptstyle \Delta\underline{\otimes} \id$} (r);
        \draw[->] (l) to[bend right=6] node[pos=.5,auto,swap] {$\scriptstyle \Delta $} (d);
        \draw[->] (d) to[bend right=6] node[pos=.5,auto,swap] {$\scriptstyle \id\underline{\otimes}\Delta$} (r);
      \end{tikzpicture}
    \end{equation*}
    commutes;

  \item and such that
    \begin{equation*}
      \mathrm{span}~\{(a\underline{\otimes} 1)\Delta(b)\ |\ a,b\in \cC(\bG)\}
      \text{ and }
      \mathrm{span}~\{(1\underline{\otimes} a)\Delta(b)\ |\ a,b\in \cC(\bG)\}
    \end{equation*}
    are dense in $\cC(\bG)^{\underline{\otimes}2}$.
  \end{itemize}
\end{definition}

To return to lifting actions to multiplier algebras, then, \Cref{th:uniffinspec} reads:

\begin{theoremN}\label{thn:a}
  A unitary representation of a compact quantum group $\bG$ on a Hilbert space $H$ induces a continuous conjugation action on $\cL(H)$ if and only if $\rho$ has finitely many isotypic components.

  In particular, said action on $\cL(H)$ will then automatically be {\it uniformly} continuous in any reasonable sense (e.g. \Cref{def:normcontcqg}\Cref{item:unifcont}).  \qedhere
\end{theoremN}

In fact, slightly more can be said for plain compact groups (\Cref{cor:actonuhunif}): their continuous actions on $C^*$-algebras of the form $\cL(H)$ automatically have finitely many isotypic components.

Action extensibility across $A\lhook\joinrel\xrightarrow{}M(A)$ also singles out a class of compact quantum groups: those for which it is always possible. By \Cref{thn:a}, these are exactly the finite ones (\Cref{th:cqgisfin}):

\begin{corollaryN}\label{corn:b}
  A compact quantum group is finite if and only if its actions on $C^*$-algebras $A$ extend to actions on the corresponding multiplier algebras $A\lhook\joinrel\xrightarrow{} M(A)$ or von Neumann envelopes $A\lhook\joinrel\xrightarrow{} A^{**}$ (the latter regarded as $C^*$-algebras again).  \qedhere
\end{corollaryN}

For compact quantum groups finiteness is equivalent to discreteness. Dropping compactness at the cost of reverting to the classical setting, \Cref{th:discr-coprod-pres} expands (the analogue of) \Cref{corn:b} into a category-theoretic characterization of discreteness for locally compact groups. A short sample of that statement:

\begin{theoremN}\label{thn:c}
  A locally compact group $\bG$ is discrete if and only if the forgetful functor 
  \begin{equation*}
    \left(\text{$\bG$-actions on compact Hausdorff spaces}\right)
    \xrightarrow{\quad\cat{forget}\quad}
    \left(\text{compact Hausdorff spaces}\right)
  \end{equation*}
  creates \cite[Definition 13.17]{ahs} coproducts.  \qedhere
\end{theoremN}

There is a natural connection between a representation's having finitely many isotypic components (in short, the finiteness of its {\it spectrum}: \Cref{def:alphaisot}\Cref{item:def:alphaisot-spec}) and its {\it uniform continuity}, already mentioned explicitly in the statement of \Cref{thn:a}. This motivates both
\begin{itemize}
\item preliminaries on compact-quantum-group actions on Banach spaces in \Cref{subse:genunif}, meant to recapture enough of the common substance of unitary representations and actions on $C^*$-algebras while keeping some of the main tools those separate constructions afford (spectral spaces);

\item and a foray into the selfsame connection between spectrum finiteness and uniform continuity for actions of {\it classical} compact groups on Banach spaces in \Cref{subse:unifcls}.  
\end{itemize}

In reference to this latter point, \Cref{th:clscpct} gathers a number of equivalent characterizations of uniform continuity for such representations/actions. The lengthy statement would be unwieldy to summarize here, but it will be worth noting that many of those characterizations, relying on sometimes-non-trivial compact-group structure results, have the following general flavor: the $\bG$-action is uniformly continuous if and only if its restriction to such and such a class of subgroups of $\bG$ is.

This last remark, in turn, sensibly raises the issue of lifting finite-spectrum properties of representations from (quantum) subgroups to larger ambient (quantum) groups; that is the direction taken in \Cref{se:topgen}, where \Cref{th:fintopgen} moves back to quantum groups, this time {\it linearly reductive} (so a broader class than compact):

\begin{theoremN}\label{thn:d}
  Let $\bG$ be a linearly reductive quantum group and $\bH,\bK\le \bG$ linearly reductive quantum subgroups that {\it generate $\bG$ topologically} (\Cref{se:topgen}, preceding \Cref{def:jointmon}), with $\bK\trianglelefteq \bG$ normal.

  A $\bG$-representation has finitely many isotypic components if and only if its restrictions $\rho|_{\bH}$ and $\rho|_{\bK}$ both do.  \qedhere
\end{theoremN}

That normality (or something like it) is needed is clear from \Cref{ex:genunif}, say. 

%%%%%%%%%%%%%%%%%%%%%%%%%%%%%%%%
\subsection*{Acknowledgements}

I am grateful for numerous comments and tips on the contents of this note at various stages, from M. Brannan, M. Daws, K.H. Hofmann, P.M. So{\l}tan, A. Skalski, N. Spronk, A. Viselter and naturally, the anonymous referee.

The work is partially supported by NSF grant DMS-2001128, and is part of the project Graph Algebras partially supported by EU grant HORIZON-MSCA- SE-2021 Project 101086394.

%%%%%%%%%%%%%%%%%%%%%%%%%%%%%%%%
%%%%%%%%%%%%%%%%%%%%%%%%%%%%%%%%
\section{Preliminaries}\label{se.prel}

The symbol `$\underline{\otimes}$' denotes the {\it minimal} (or {\it spatial} $C^*$-algebra tensor product \cite[\S 12]{pls-bk}, \cite[Definition IV.4.8]{tak1}, etc.). The symbol `$\otimes$' denotes the plain algebraic tensor product when one of the tensorands is not equipped with any analytic structure. We write
\begin{itemize}

\item $\cL(E,F)$ for bounded linear maps between Banach spaces $E$ and $F$, with $\cL(E,E)$ abbreviated to $\cL(E)$.

\item $\cK(E)$ for compact operators on a Banach (hence also Hilbert) space $E$.
  
\item $M(A)$ for the {\it multiplier algebra} \cite[\S II.7.3]{blk} of a $C^*$-algebra $A$.
  
\item $\cO(\bG)\le \cC(\bG)$ for the unique dense Hopf $*$-subalgebra (the $\cA$ of \cite[Theorem 1.2]{wor-cqg} and $\cA=C[G]$ in \cite[Theorem 1.6.7]{NeTu13}). These are {\it CQG algebras} \cite[Definition 2.2]{dk_cqg}: particularly well-behaved {\it cosemisimple} \cite[Definition 2.4.1]{mont} complex Hopf $*$-algebras.

\item $\Irr(\bG)$ for the (isomorphism classes of) unitary irreducible $\bG$-representations. We will often apply tensor operations directly to the symbols standing for such classes, with `$\le$' denoting the relation of being a summand. Thus:
  \begin{equation*}
    \gamma\le \alpha\otimes\beta,\quad \alpha,\beta,\gamma\in\Irr(\bG)
  \end{equation*}
  means that tensoring representations of classes $\alpha$ and $\beta$ produces a $\gamma$ summand, $\alpha^*$ denotes the representation dual (or {\it contragredient} \cite[Definition 1.3.8]{NeTu13}) to $\alpha$, etc.
  
\item $L^2(\bG)$ for the Hilbert space carrying the GNS representation associated to the {\it Haar state} \cite[\S 1.2]{NeTu13} $h\in \cC(\bG)^*$.
 
  % % \item More generally, $L^p(\bG)$, $1\le p\le \infty$ for the {\it non-commutative $L^p$ space} associated to the Haar state, as usually done (e.g. \cite[\S IX.2, Theorem 2.13]{tak2}): the completion of $\cC(\bG)$ for the seminorm
  % %   \begin{equation*}
  % %     \|x\|_p:=h\left(|x|^p\right)^{\frac 1p}
  % %     ,\quad
  % %     |x|:=\text{absolute value \cite[\S I.4.2.3]{blk} of $x$}. 
  % %   \end{equation*}
  % %   
  
\end{itemize}

Since in much of the sequel $\cO(\bG)$ plays rather a more important role than $\cC(\bG)$, the phrase {\it (compact) quantum subgroup} consistently refers to Hopf $*$-algebra quotients $\cO(\bG)\to \cO(\bH)$. 

%\newpage

%%%%%%%%%%%%%%%%%%%%%%%%%%%%%%%%
%%%%%%%%%%%%%%%%%%%%%%%%%%%%%%%%
\section{Uniform CQG actions on Banach spaces}\label{se:actban}

%%%%%%%%%%%%%%%%%%%%%%%%%%%%%%%%
\subsection{Generalities on representations and uniformity}\label{subse:genunif}

Let $\bG$ be a compact quantum group. The appropriate uniform context for analyzing both unitary $\bG$-representations and $\bG$-actions on $C^*$-algebras, both of which feature extensively in the literature, seems to be the following setup:

\begin{definition}\label{def:cqgactban}
  Let $E$ be a Banach space (equipped with a norm we typically suppress notationally) and $\bG$ a compact quantum group. A {\it $\bG$-action (or $\bG$-representation) on $E$} is a continuous morphism
  \begin{equation}\label{eq:actone}
    E
    \xrightarrow{\quad\rho\quad}
    \cC(\bG)\otimes_{\varepsilon}E
  \end{equation}
  into the {\it injective (or minimal)} \cite[\S 3.1]{ryan_ban} Banach-space tensor product such that
  \begin{enumerate}[(a)]
  \item\label{item:coassoc} the diagram
    \begin{equation}\label{eq:coassocban}
      \begin{tikzpicture}[auto,baseline=(current  bounding  box.center)]
        \path[anchor=base] 
        (0,0) node (l) {$E$}
        +(2,.5) node (u) {$\cC(\bG)\otimes_{\varepsilon} E$}
        +(2,-.5) node (d) {$\cC(\bG)\otimes_{\varepsilon} E$}
        +(6,0) node (r) {$\cC(\bG)\otimes_{\varepsilon}\cC(\bG)\otimes_{\varepsilon} E$}
        ;
        \draw[->] (l) to[bend left=6] node[pos=.5,auto] {$\scriptstyle \rho$} (u);
        \draw[->] (u) to[bend left=6] node[pos=.5,auto] {$\scriptstyle (\Delta\otimes\id)\rho$} (r);
        \draw[->] (l) to[bend right=6] node[pos=.5,auto,swap] {$\scriptstyle \rho$} (d);
        \draw[->] (d) to[bend right=6] node[pos=.5,auto,swap] {$\scriptstyle (\id\otimes\rho)\rho$} (r);
      \end{tikzpicture}
    \end{equation}
    commutes;

  \item\label{item:dense} and
    \begin{equation*}
      \overline{
        \mathrm{span}\left\{(x\otimes\id)\rho(v)\ |\ x\in \cC(\bG),\ v\in E\right\}
      }
      =
      \cC(\bG)\otimes_{\varepsilon}E.
    \end{equation*}
  \end{enumerate}  
\end{definition}

Note that \Cref{eq:coassocban} does indeed make sense, given that
\begin{itemize}
\item the minimal $C^*$ tensor product $\underline{\otimes}$ into which the comultiplication
  \begin{equation*}
    \cC(\bG)\xrightarrow{\quad\Delta\quad}\cC(\bG)\underline{\otimes}\cC(\bG)
  \end{equation*}
  is assumed to take values dominates the injective Banach-space norm (e.g. by \cite[Proposition IV.2.2 and discussion preceding Theorem IV.4.9]{tak1}) and hence can also be regarded as a map into $\cC(\bG)\otimes_{\varepsilon}\cC(\bG)$;

\item and the injective Banach-space tensor product is functorial \cite[Proposition 3.2]{ryan_ban} (with respect to continuous linear maps between Banach spaces), so the various compositions are well defined.
\end{itemize}

The {\it spectral decomposition} results familiar from the theory of CQG actions on $C^*$-algebras transport over to the Banach-space setup. To make sense of this, recall the idempotent operators
\begin{equation*}
  E\xrightarrow{\quad \pi^{\alpha}\quad}E,\quad \alpha\in \Irr(\bG)
\end{equation*}
employed in the proof of \cite[Theorem 1.5]{podl_symm}, whose images are, respectively, the {\it $\alpha$-spectral subspaces} attached to the action:

\begin{itemize}
\item One first defines appropriate continuous functionals $\rho^{\alpha}$ (\cite[p.657]{wor}, \cite[p.3]{podl_symm}) on $\cC(\bG)$ vanishing on the spaces of coefficients associated to $\beta\ne \alpha\in\Irr(\bG)$. In the notation of \cite{podl_symm},
  \begin{equation}\label{eq:rhoalpha}
    \rho^{\alpha}(u^{\beta}_{ij}) = \delta_{\alpha\beta}\delta_{ij}
    ,\quad
    \delta_{\cdot}:=\text{Kronecker delta}.
  \end{equation}

\item The idempotent $\pi^{\alpha}$ can then be defined by (the commutativity of)
  \begin{equation}\label{eq:pirho}
    \begin{tikzpicture}[auto,baseline=(current  bounding  box.center)]
      \path[anchor=base] 
      (0,0) node (l) {$E$}
      +(2,.5) node (u) {$\cC(\bG)\otimes E$}
      +(4,0) node (r) {$E$.}
      ;
      \draw[->] (l) to[bend left=6] node[pos=.5,auto] {$\scriptstyle $} (u);
      \draw[->] (u) to[bend left=6] node[pos=.5,auto] {$\scriptstyle \rho^{\alpha}\otimes\id$} (r);
      \draw[->] (l) to[bend right=6] node[pos=.5,auto,swap] {$\scriptstyle \pi^{\alpha}$} (r);
    \end{tikzpicture}
  \end{equation}    
\end{itemize}
All of this goes through for $\otimes_{\varepsilon}$ as well as it does for $\underline{\otimes}$, hence. In particular, the proof of \cite[Theorem 1.5]{podl_symm} delivers

\begin{theorem}\label{th:spectralenough}
  Given an action \Cref{eq:actone} of a compact quantum group on a Banach space, the (algebraic) direct sum
  \begin{equation}\label{eq:gfinvects}
    \bigoplus_{\alpha\in\Irr(\bG)}\left(E^{\alpha}:=\mathrm{im}~\pi^{\alpha}\right)\subseteq E
  \end{equation}
  is dense.  \qedhere
\end{theorem}

\begin{definition}\label{def:alphaisot}
  \begin{enumerate}[(1)]
  \item\label{item:def:alphaisot-isotyp} We will refer to the subspace $E^{\alpha}\le E$ of \Cref{th:spectralenough} as the {\it $\alpha$-isotypic component} of the action (or of $E$, slightly abusively).
    
    They are closed and in fact {\it complemented} \cite[\S III.13]{conw_fa} subspaces of $E$, being ranges of continuous idempotent operators $\pi^{\alpha}$.

  \item\label{item:def:alphaisot-spec} The set of $\alpha$ for which $E^{\alpha}\ne \{0\}$ is the {\it spectrum} of $\rho$.

  \item\label{item:def:alphaisot-gfin} The elements of the algebraic direct cum on the left-hand side of \Cref{eq:gfinvects} are the {\it $\bG$-finite} elements of $E$ (with respect to $\rho$).

    This transports the terminology of \cite[Definition 3.1]{hm4} into the present quantum setting.
  \end{enumerate}
\end{definition}

When $E$ is additionally a unital $C^*$-algebra $A$, it seems reasonable to ask of an action that it be a $C^*$-morphism 
\begin{equation*}
  A\to \cC(\bG)\underline{\otimes} A,
\end{equation*}
as one usually does \cite[Definition 1.4]{podl_symm}. The other constraints will then follow, however, from the apparently weaker analogous ones of \Cref{def:cqgactban}; this, despite the appearance of the minimal {\it Banach}-space tensor product, which is nowhere featured in \cite[Definition 1.4]{podl_symm}.

\begin{proposition}\label{pr:sameactoncastalg}
  Let $\bG$ be a compact quantum group and $A$ a unital $C^*$-algebra. A $C^*$-morphism
  \begin{equation*}
    A\xrightarrow{\quad\alpha\quad}\cC(\bG)\underline{\otimes} A
  \end{equation*}
  is an action in the sense of \cite[Definition 1.4]{podl_symm} if and only if the composition
  \begin{equation*}
    \begin{tikzpicture}[auto,baseline=(current  bounding  box.center)]
      \path[anchor=base] 
      (0,0) node (l) {$A$}
      +(2,.5) node (u) {$\cC(\bG)\underline{\otimes} A$}
      +(5,0) node (r) {$\cC(\bG)\otimes_{\varepsilon} A$}
      ;
      \draw[->] (l) to[bend left=6] node[pos=.5,auto] {$\scriptstyle \alpha$} (u);
      \draw[->] (u) to[bend left=6] node[pos=.5,auto] {$\scriptstyle $} (r);
      \draw[->] (l) to[bend right=6] node[pos=.5,auto,swap] {$\scriptstyle \rho$} (r);
    \end{tikzpicture}
  \end{equation*}
  is an a $\bG$-action on the Banach space $A$ in the sense of \Cref{def:cqgactban}. 
\end{proposition}
\begin{proof}
  One implication is clear: the density and coassociativity of the original $\alpha$ entail the analogous properties of $\rho$, given that $\underline{\otimes}$ surjects onto $\otimes_{\varepsilon}$. It is thus the converse that is of interest, and our focus for the duration of the proof.

  % % To that end, recall the idempotent operators
  % % \begin{equation*}
  % %   A\xrightarrow{\quad \pi^{\alpha}\quad}A,\quad \alpha\in \Irr(\bG)
  % % \end{equation*}
  % % employed in the proof of \cite[Theorem 1.5]{podl_symm}, whose images are, respectively, the {\it $\alpha$-spectral subspaces} attached to the action:
  % % \begin{itemize}
  % % \item One first defines appropriate continuous functionals $\rho^{\alpha}$ (\cite[p.657]{wor}, \cite[p.3]{podl_symm}) on $\cC(\bG)$ vanishing on the spaces of coefficients associated to $\beta\ne \alpha\in\Irr(\bG)$;
  % % \item then defining $\pi^{\alpha}$ by (the commutativity of)
  % %   \begin{equation*}
  % %     \begin{tikzpicture}[auto,baseline=(current  bounding  box.center)]
  % %       \path[anchor=base] 
  % %       (0,0) node (l) {$A$}
  % %       +(2,.5) node (u) {$\cC(\bG)\otimes A$}
  % %       +(4,0) node (r) {$A$.}
  % %       ;
  % %       \draw[->] (l) to[bend left=6] node[pos=.5,auto] {$\scriptstyle $} (u);
  % %       \draw[->] (u) to[bend left=6] node[pos=.5,auto] {$\scriptstyle \id\otimes\rho^{\alpha}$} (r);
  % %       \draw[->] (l) to[bend right=6] node[pos=.5,auto,swap] {$\scriptstyle \pi^{\alpha}$} (r);
  % %     \end{tikzpicture}
  % %   \end{equation*}    
  % % \end{itemize}
  % % All of this goes through for $\otimes_{\varepsilon}$ as well as it does for $\underline{\otimes}$, so that the proof of \cite[Theorem 1.5]{podl_symm} will show that the span of
  % % 

  It follows from \Cref{th:spectralenough} that the span of
  \begin{equation*}
    A^{\alpha}:=\mathrm{im}~\pi^{\alpha},\quad \alpha\in\Irr(\bG)
  \end{equation*}
  is a dense $*$-subalgebra in $A$. This then recovers an action in the sense of \cite[Definition 1.4]{podl_symm} by \cite[Corollary 1.6]{podl_symm}, finishing the proof.
\end{proof}

As for Hilbert spaces $H$, recall (e.g. \cite[Definition 2.1]{dsv}) that a {\it unitary representation} of a compact quantum group $\bG$ on a Hilbert space $H$ is a unitary
\begin{equation}\label{eq:uonh}
  U\in M(\cC(\bG)\underline{\otimes} \cK(H))
\end{equation}
with
\begin{equation}\label{eq:u1323}
  (\Delta\otimes\id)U = U_{13}U_{23}
\end{equation}
in the {\it leg-numbering notation} of \cite[\S 1.3]{NeTu13}. These too are examples of Banach-space $\bG$-actions in the sense of \Cref{def:cqgactban}:

\begin{proposition}\label{pr:actonhilb}
  A unitary representation \Cref{eq:uonh} of a compact quantum group $\bG$ on a Hilbert space $H$ is induces a $\bG$-action on the Banach space $H$ in the sense of \Cref{def:cqgactban} via
  \begin{equation}\label{eq:hgh}
    H\ni v\xmapsto{\quad\rho\quad}U^*(1\otimes v)\in \cC(\bG)\otimes_{\varepsilon} H.
  \end{equation}
\end{proposition}
\begin{proof}
  \Cref{eq:hgh} perhaps requires some unwinding. Recall from \cite[discussion following Lemma 2.5]{dsv} that a unitary representation \Cref{eq:uonh} (or its adjoint $U^*$) can be recast as an automorphism of the Hilbert $\cC(\bG)$-module $\cC(\bG)\boxtimes H$ obtained as the {\it exterior tensor product} \cite[\S 3, pp.34-35]{lnc-hilb} of $\cC(\bG)$ regarded as a Hilbert module over itself and the Hilbert space (i.e. $\bC$-Hilbert module) $H$ (this is also what on \cite[p.6]{lnc-hilb} would be denoted simply by $\cC(\bG)\otimes H$).

  Since the Hilbert-module norm resulting giving rise to the exterior tensor product is easily seen to dominate the injective Banach-space norm, \Cref{eq:hgh} is the composition
  \begin{equation*}
    \begin{tikzpicture}[auto,baseline=(current  bounding  box.center)]
      \path[anchor=base] 
      (0,0) node (l) {$H$}
      +(2,.5) node (u) {$\cC(\bG)\boxtimes H$}
      +(5,0) node (r) {$\cC(\bG)\otimes_{\varepsilon} H$}
      ;
      \draw[right hook->] (l) to[bend left=6] node[pos=.5,auto] {$\scriptstyle$} (u);
      \draw[->] (u) to[bend left=6] node[pos=.5,auto] {$\scriptstyle $} (r);
      \draw[->] (l) to[bend right=6] node[pos=.5,auto,swap] {$\scriptstyle$} (r);
    \end{tikzpicture}
  \end{equation*}
  
  The coassociativity requirement of \Cref{def:cqgactban} is, in this case, a reformulation of
  \begin{equation*}
  (\Delta\otimes\id)U^* = U_{23}^*U_{13}^*,
\end{equation*}
which in turn rephrases \Cref{eq:u1323} (this is why the `$*$' is necessary: note the order reversal between the `13' and `23' subscripts). 

As for the density condition \Cref{item:dense} of \Cref{def:cqgactban}, it follows from the same discussion preceding \cite[\S 2.3]{dsv}: as explained there,
\begin{equation*}
  (x\otimes\id)\rho(v),\quad x\in \cC(\bG),\ v\in E
\end{equation*}
(with $\rho$ as in \Cref{eq:hgh}) is precisely the image of
\begin{equation*}
  x\boxtimes v\in \cC(\bG)\boxtimes H
\end{equation*}
through the isomorphism of that Hilbert module associated to $U$, followed by the surjection
\begin{equation*}
  \cC(\bG)\boxtimes H\xrightarrow{\quad} \cC(\bG)\otimes_{\varepsilon} H.
\end{equation*}
Naturally, then, the span of such elements is dense.
\end{proof}

The reason why $M(\cC(\bG)\underline{\otimes} \cK(H))$ is a good ambient space for the unitary representation $U$ of \Cref{eq:uonh} is that classically (i.e. for ordinary compact $\bG$) that space can be identified \cite[Corollary 3.4]{apt} with
\begin{equation*}
  \left\{\text{bounded functions }\bG\to \cL(H)\cong M(\cK(H)),\ \text{continuous for the strict topology}\right\}.
\end{equation*}
On the unitary subgroup $\bU(H)$ of $B(\cH)\cong M(\cK(H))$ the strict topology coincides \cite[Proposition I.3.2.9]{blk} with all of the various ``weak'' topologies of \cite[\S II.2]{tak1}, in particular with the {\it strong topology} induced by the norms
\begin{equation*}
  \bU(H)\ni U\xmapsto{\quad}\|U\xi\|,\quad \xi\in H.  
\end{equation*}
The containment \Cref{eq:uonh}, cast in terms of a representation $\bG\xrightarrow{\varphi} \bU(H)$, thus means precisely the usual requirement (e.g. \cite[\S 2]{rob}) that
\begin{equation*}
  \bG\ni g\xmapsto{\quad} \varphi(g)\xi\in H
\end{equation*}
be continuous for every $\xi\in H$. 

One is occasionally interested in representations exhibiting the stronger requirement of {\it uniform} continuity \cite[\S 8.5]{ped-aut}: the requirement that $\bG\xrightarrow{\varphi}\bU=\bU(H)$ be continuous for the {\it norm} topology on $\bU$. The present line of reasoning justifies the following setup. 

\begin{definition}\label{def:normcontcqg}
  \begin{enumerate}[(1), leftmargin=*, wide=0pt]
  \item\label{item:equicont} For two Banach spaces $E$ and $F$, a subset of $F\otimes_{\varepsilon}E$ is {\it $F$- or (left-)$w^*$-equicontinuous} if it is equicontinuous \cite[Definition VI.3.7]{conw_fa} on the dual $F^*$ equipped with its weak$^*$-topology \cite[Example IV.1.8]{conw_fa} upon embedding
    \begin{equation}\label{eq:fef*e}
      F\otimes_{\varepsilon} E\le \cL(F^*,E)
    \end{equation}
    isometrically \cite[Proposition IV.2.1]{tak1}.

  \item\label{item:bddp} Given a property $\cP$ that a set may or may not have (assuming also phrasing to the effect that sets {\it are} $\cP$), a linear map $E\xrightarrow{u}F$ between topological vector spaces is {\it boundedly $\cP$} if it sends bounded \cite[\S 15.6]{koeth_tvs-1} sets to sets with property $\cP$. 
    
  \item\label{item:unifcont} A representation \Cref{eq:actone} of a compact quantum group $\bG$ on a Banach space $E$ is {\it uniformly continuous} (occasionally {\it uniform} for short) if it is boundedly left-$w^*$-equicontinuous in the sense of point \Cref{item:bddp}: it maps bounded subsets of $E$ into left-$w^*$-equicontinuous subsets of $\cC(\bG)\otimes_{\varepsilon} E$.

    Equivalently, it is enough to require this only for the unit ball of $E$ (or any single bounded origin neighborhood).
  \end{enumerate}
\end{definition}

A simple reformulation of \Cref{def:normcontcqg}, more in keeping with the spirit of the requirement for a classical $\bG$-action on $E$ that the corresponding map $\bG\to \cL(E)$ be continuous for the norm topology:

\begin{lemma}\label{le:unifcont-alt}
  A $\bG$-action \Cref{eq:actone} on a Banach space is uniformly continuous in the sense of \Cref{def:normcontcqg}\Cref{item:unifcont} if and only if the map
  \begin{equation}\label{eq:rhobar}
    \cC(\bG)^*\xrightarrow{\quad\overline{\rho}\quad}\cL(E)
  \end{equation}
  defined by
  \begin{equation*}
    \overline{\rho}(f)(v) := (f\otimes\id)\rho(v),\quad f\in \cC(\bG)^*,\quad v\in E
  \end{equation*}
  is weak$^*$-to-norm continuous.
\end{lemma}
\begin{proof}
  Composed with the embedding \Cref{eq:fef*e} (for $F:=\cC(\bG)$), the original action structure map \Cref{eq:actone} gives a (bilinear) pairing
  \begin{equation*}
    \cC(\bG)^*\times E\xrightarrow{\quad\beta\quad} E. 
  \end{equation*}
  Unpacked, the equicontinuity condition of \Cref{def:normcontcqg} means that a weak$*$-convergent net
  \begin{equation*}
    f_{\lambda}\xrightarrow[\quad\lambda\quad]{\text{weak$^*$}}f\in \cC(\bG)^*
  \end{equation*}
  results in a net 
  \begin{equation*}
    \beta(f_{\lambda},-)=\overline{\rho}(f_{\lambda})\xrightarrow[\quad\lambda\quad]{}\overline{\rho}(f)=\beta(f,-)\in \cL(E)
  \end{equation*}
  converging uniformly on bounded subsets of $E$ (or equivalently, on the unit ball of $E$). This being exactly the requirement in the present statement, we are done. 
\end{proof}

As perhaps familiar \cite[Theorem 8.1.12]{ped-aut} from the classical story on uniformly continuous actions of locally compact abelian groups, the condition (of uniform continuity) should have something to do with the compactness of the spectrum of the action. One direction is a simple observation, and worth recording here for future reference:

\begin{lemma}\label{le:unif2finsp}
  If an action \Cref{eq:actone} of a compact quantum group on a Banach space has finite spectrum, then it is uniformly continuous. 
\end{lemma}
\begin{proof}
  Immediate by \Cref{le:unifcont-alt}, once we observe that under the hypothesis the map \Cref{eq:rhobar} factors through the span of functionals $\rho^{\alpha}_{k\ell}$ with
  \begin{equation*}
    \rho^{\alpha}_{k\ell}(u^{\beta}_{ij}) = \delta_{\alpha\beta}\delta_{ki}\delta_{\ell j}
  \end{equation*}
  (as on \cite[p.3]{podl_symm}, analogous to \Cref{eq:rhoalpha}), for finitely many $\alpha$. That span is of course finite-dimensional, so all reasonable vector space topologies on it coincide (hence the continuity required by \Cref{le:unifcont-alt}).
\end{proof}

Motivated by the fact that for a classical compact group the uniform continuity of a unitary representation could also be phrased as \Cref{eq:uonh} belonging to $\cC(\bG)\underline{\otimes} \cL(H)$ (see also \Cref{th:clscpct} below), we also note explicitly:

\begin{remark}\label{re:uisintens}
  A finite-spectrum action \Cref{eq:actone} is implemented by an operator in the algebraic tensor product $\cC(\bG)\otimes \cL(E)$. 
\end{remark}

%%%%%%%%%%%%%%%%%%%%%%%%%%%%%%%%
\subsection{Plain compact groups}\label{subse:unifcls}

Before returning to {\it quantum} (mostly compact) groups, we first address this in the classical case, for two reasons:

\begin{itemize}
\item anchored, as they are, in compact-group structure theory, the proof techniques do not (all) transport over to the quantum setting, so they might be of some independent interest;

\item and the main statement itself does not {\it fully} generalize well to quantum groups, given that it relies rather heavily on point-topological notions.
\end{itemize}

In reference to the second point, recall (\cite[p.2, Definition 1 and Yamabe's Theorem following it]{hm_pro-lie-bk}; also \cite[\S 4.6, Theorem]{mz} for the latter) that a topological group $\bG$ is
\begin{itemize}
\item  {\it almost-connected} if its quotient $\bG/\bG_0$ by the connected component of the identity is compact;

\item {\it pro-Lie} if it is the inverse limit of its (finite-dimensional) Lie-group quotients;

\item and pro-Lie as soon as it is locally compact and almost-connected. 
\end{itemize}

In particular, compact groups are inverse ({\it cofiltered} \spr{04AY}) limits of their Lie quotients \cite[Corollary 2.36]{hm4}. Recall also \cite[Definitions 9.30]{hm4} that {\it pro-tori} are compact connected abelian groups. Since the compact connected abelian {\it Lie} groups are precisely \cite[Corollary 4.2.6]{de} the tori $\bT^n\cong (\bS^1)^n$, the nomenclature is apposite: the pro-tori are exactly the cofiltered limits of tori. Similarly, {\it pro-$p$ groups} (for a prime $p$) are \cite[preceding Theorem 1.2.3, p.18]{wils_prof_bk} cofiltered limits of finite $p$-groups. 

%\newpage % THM: START

\begin{theorem}\label{th:clscpct}
  Given a representation
  \begin{equation}\label{eq:gonrho}
    \bG\xrightarrow{\quad\rho\quad} GL(E),\quad E\text{ a Banach space}
  \end{equation}
  of a compact group, the following conditions are equivalent:
  
  \begin{enumerate}[(a)]
  \item\label{item:clscac-fin} $\rho$ has finite spectrum (i.e. finitely many non-zero isotypic components: \Cref{def:alphaisot}\Cref{item:def:alphaisot-spec}). 

  \item\label{item:aprime} All elements of $E$ are {\it locally finite}, in the sense that their $\bG$-orbits span finite-dimensional spaces. 
    
  \item\label{item:clscac-unif} $\rho$ is uniformly continuous in the sense of \Cref{def:normcontcqg}\Cref{item:unifcont}.

  \item\label{item:clscac-g2gl} Regarded as a map \Cref{eq:actone}, $\rho$ is implemented by an operator $U\in \cC(\bG,GL(E))\subset \cC(\bG,\cL(E))$. 

  \item\label{item:clscac-clim} The image $\rho(\bG)$ is a compact (hence finite-dimensional Lie) subgroup of $GL(E)$

  \item\label{item:clscac-allsubgp} The equivalent conditions \Cref{item:clscac-fin} to \Cref{item:clscac-clim} hold for (the restriction of $\rho$ to) every closed subgroup of $\bG$.
    
  \item\label{item:clscac-g0profin} \Cref{item:clscac-fin}-\Cref{item:clscac-clim} hold for the identity component $\bG_0$ and the profinite subgroups of $\bG$. 
    
  \item\label{item:clscac-maxab} \Cref{item:clscac-fin}-\Cref{item:clscac-clim} hold for every (maximal) closed abelian subgroup of $\bG_0$ and every pro-$p$ subgroup of $\bG$.

  \item\label{item:clscac-toriprof} \Cref{item:clscac-fin}-\Cref{item:clscac-clim} hold for every (maximal) pro-torus and every pro-$p$ subgroup of $\bG$.

  \item\label{item:clscac-singletor} \Cref{item:clscac-fin}-\Cref{item:clscac-clim} hold for a single maximal pro-torus and every pro-$p$ subgroup of $\bG$.

  \item\label{item:clscac-singletor-abim} \Cref{item:clscac-fin}-\Cref{item:clscac-clim} hold for a single maximal pro-torus and every pro-$p$ subgroup of $\bG$ whose image under $\rho$ is abelian.
  \end{enumerate}  
\end{theorem}
\begin{proof}
  Locally compact subgroups of Banach Lie groups are automatically finite-dimensional and Lie \cite[Theorem IV.3.16]{neeb-lc}, hence the parenthetic remark in \Cref{item:clscac-clim}. We address the various claims separately.
  \begin{enumerate}[label={}, leftmargin=*, wide=0pt]

  \item {\bf \Cref{item:clscac-fin} $\xLeftrightarrow{}$ \Cref{item:aprime}:} The interesting implication is `$\Leftarrow$'. To verify it, suppose $\rho$ contains infinitely many non-isomorphic irreducible summands $\rho_n$, $n\in \bZ_{>0}$ supported on respective subspaces $E_i\le E$. A vector of the form
    \begin{equation*}
      v=\sum_n v_n,\quad v_n\in E_n,\quad \sum\|v_n\|<\infty
    \end{equation*}
    will then fail to be locally finite. 
    
  \item {\bf \Cref{item:clscac-fin} $\xRightarrow{}$ \Cref{item:clscac-unif}:} This follows from \Cref{le:unif2finsp}. 

  \item {\bf \Cref{item:clscac-unif} $\xLeftrightarrow{}$ \Cref{item:clscac-g2gl}:} This is tautological, and included only for further reference: uniform continuity in any sense made explicit in \Cref{def:normcontcqg} is precisely intended to mimic norm-continuous maps $\bG\xrightarrow{} GL(E)$. 
    
  \item {\bf \Cref{item:clscac-unif} $\xRightarrow{}$ \Cref{item:clscac-clim}:} Naturally, since $\rho$ is a continuous map defined on a compact space, so its image is compact \cite[Theorem 17.7]{wil-top}. 

  \item {\bf \Cref{item:clscac-clim} $\xRightarrow{}$ \Cref{item:clscac-fin} assuming \Cref{item:clscac-unif} $\xRightarrow{}$ \Cref{item:clscac-fin}:} Given \Cref{item:clscac-clim}, the hypothesis \Cref{item:clscac-unif} applies to the compact group $\rho(\bG)$. Since we are also assuming that \Cref{item:clscac-unif} implies \Cref{item:clscac-fin}, it follows that $E$ decomposes as a finite sum of isotypic components over the subgroup
    \begin{equation}\label{eq:rhogingl}
      \rho(\bG)\le GL(E)
    \end{equation}
    But then it does so as a $\bG$-representation via $\rho$, since that representation (by definition) factors through \Cref{eq:rhogingl}.      

  \item {\bf \Cref{item:clscac-unif} $\xRightarrow{}$ \Cref{item:clscac-fin}:} The general linear group $GL(E)$ has neighborhoods containing no non-trivial subgroups: this is the celebrated {\it no-small-subgroups} property, valid for all Banach Lie groups (\cite[Theorem III.2.3]{neeb-lc}, citing \cite[Theorem 1.4.2]{om-transf}). The assumed uniform continuity of $\rho$ then implies, by the above-cited \cite[Corollary 2.36]{hm4}, that $\rho$ factors through a Lie quotient of $\bG$. That (compact Lie) quotient will then have finitely many components, so the target implication reduces to compact connected Lie groups. This, then, will be the standing assumption on $\bG$ for the duration of the proof.

    Compact connected Lie groups are, up to {\it isogenies} (i.e. quotients by finite central subgroups), products of tori and compact {\it semisimple} Lie groups \cite[Theorem 6.19]{hm4}. If $\bT\le \bG$ is a maximal torus, the classification \cite[Theorem IX.5.1]{simon_fincpct} of the irreducible representations of a compact semisimple Lie group in terms of highest weights makes it clear that any infinite set of (isomorphism classes of) $\bG$-representations will restrict to infinitely many characters of $\bT$. It is thus enough to prove the claim for $\bT$, and hence for the circle group $\bS^1$ ($\bT$ being a product of copies of the latter).

    But for $\bS^1$ the claim is immediate: the characters are
    \begin{equation*}
      \bS^1\ni z\xmapsto{\quad\chi_n\quad}z^n\in \bS^1\subset \bC^{\times},\quad n\in \bZ,
    \end{equation*}
    and no {\it infinite} direct sum of such is uniformly continuous. Indeed, suppose the representation admits norm-1 $\chi_{n_{\ell}}$-eigenvectors $v_{\ell}\in E$ for a sequence
    \begin{equation*}
      n_{\ell}\in \bZ,\quad |n_{\ell}|\xrightarrow[\ell]{\quad}\infty.
    \end{equation*}
    We can then find $z_{\ell}\in\bS^1$ approaching $1\in \bS^1$ with $z_{\ell}^{n_{\ell}}$ uniformly far from $1$, say
    \begin{equation*}
      |z_{\ell}^{n_{\ell}}-1|\ge C>0;
    \end{equation*}
    one example: $z_{\ell}:=\exp(\pi i / n_{\ell})$ will do, with $C=\sqrt 2$. It follows that
    \begin{equation*}
      \|\rho(z_{\ell})v_{\ell}-v_{\ell}\| = \|\chi_{n_{\ell}}(z_{\ell})v_{\ell}-v_{\ell}\| = |z_{\ell}^{n_{\ell}}-1| \cdot\|v_{\ell}\|\ge C>0,\quad \forall i.
    \end{equation*}
    In particular $\rho(z_{\ell})$ does not converge to $1$, contradicting uniform continuity.
  \end{enumerate}
  This, thus far, completes the proof of the mutual equivalence
  \begin{equation*}
    \text{
      \Cref{item:clscac-fin}
      $\xLeftrightarrow{\quad}$
      \Cref{item:aprime}
      $\xLeftrightarrow{\quad}$
      \Cref{item:clscac-unif}
      $\xLeftrightarrow{\quad}$
      \Cref{item:clscac-clim}}
  \end{equation*}
  for arbitrary compact groups; we can thus treat the three conditions as interchangeable in the subsequent discussion.

  \begin{enumerate}[label={}, leftmargin=*, wide=0pt]  

  \item {\bf \Cref{item:clscac-fin} $\xRightarrow{}$ \Cref{item:clscac-allsubgp}:} Since every irreducible $\bG$-representation is finite-dimensional \cite[Theorem 3.51]{hm4}, the finitely many appearing in the decomposition of $\rho$ restrict to finitely many over any closed $\bH\le \bG$.

  \item {\bf \Cref{item:clscac-allsubgp} $\xRightarrow{}$ \Cref{item:clscac-g0profin}:} Obvious.

  \item {\bf \Cref{item:clscac-g0profin} $\xRightarrow{}$ \Cref{item:clscac-clim}:} \cite[Theorem 9.41]{hm4} ensures the existence of a profinite subgroup $\bD\le \bG$ that ``almost'' supplements the identity component, in the sense that $\bG=\bG_0\cdot \bD$. Since $\bG_0\trianglelefteq \bG$ is normal, it is normalized by $\bD$ and hence
    \begin{equation*}
      \rho(\bG) = \rho(\bG_0)\cdot \rho(\bD)\le GL(E).
    \end{equation*}
    In particular, $\rho(\bG)$ is compact if and only if $\rho(\bD)$ and $\rho(\bG_0)$ both are.     
  \end{enumerate}    
  We now have      
    \begin{equation}\label{eq:a2f}
      \text{
        \Cref{item:clscac-fin}
        $\xLeftrightarrow{\quad}$
        \Cref{item:aprime}
        $\xLeftrightarrow{\quad}$
        \Cref{item:clscac-unif}
        $\xLeftrightarrow{\quad}$
        \Cref{item:clscac-clim}
        $\xLeftrightarrow{\quad}$
        \Cref{item:clscac-allsubgp}
        $\xLeftrightarrow{\quad}$
        \Cref{item:clscac-g0profin}
      },
    \end{equation}
    and proceed with the other implications (involving restrictions to {\it (pro-)nilpotent} subgroups). 

    \begin{enumerate}[label={}, leftmargin=*, wide=0pt]  
      
    \item {\bf \Cref{item:clscac-allsubgp} $\xRightarrow{}$ \Cref{item:clscac-maxab} $\xRightarrow{}$ \Cref{item:clscac-toriprof} $\xLeftrightarrow{}$ \Cref{item:clscac-singletor}:} The forward implications are all formal, since each subsequent condition recapitulates the former for a smaller class of groups.

      It is also clear that maximality is optional in both \Cref{item:clscac-maxab} and \Cref{item:clscac-toriprof} (i.e. omitting it does not affect the strength of either condition), since every closed (connected) abelian subgroup of $\bG$ is contained in a maximal such (Zorn's Lemma \cite[Theorem 5.4]{jech_st}). 

      Finally, the single backward implication \Cref{item:clscac-toriprof} $\xLeftarrow{}$ \Cref{item:clscac-singletor} follows from the fact that maximal pro-tori are all conjugate \cite[Theorem 9.32 (i)]{hm4}.

    \item {\bf \Cref{item:clscac-singletor} $\xRightarrow{}$ \Cref{item:clscac-fin} ($\bG$ connected):} In this case it is in fact sufficient to demand \Cref{item:clscac-fin} (or \Cref{item:clscac-unif}, or \Cref{item:clscac-clim}) for a maximal pro-torus (i.e. the profinite arm of the condition is not needed). Indeed, this is effectively what the proof of the implication \Cref{item:clscac-unif} $\xRightarrow{}$ \Cref{item:clscac-fin} given above in fact shows.

    \item {\bf \Cref{item:clscac-singletor} $\xRightarrow{}$ \Cref{item:clscac-fin} (reduction to profinite groups):} The preceding argument, addressing the connected-$\bG$ case, already proves that the restriction of $\rho|_{\bG_0}$ to the identity connected component $\bG_0\le \bG$ satisfies the requisite conditions. We can then indeed focus on a profinite subgroup of $\bG$ by \Cref{eq:a2f} (specifically: \Cref{item:clscac-g0profin} $\xRightarrow{}$ \Cref{item:clscac-fin}).
      
    \item {\bf \Cref{item:clscac-singletor} $\xRightarrow{}$ \Cref{item:clscac-fin} (conclusion):} The preceding step allows us to assume that $\bG$ is profinite to begin with. Now, regardless of the uniformity of $\rho$, $\ker\rho\trianglelefteq \bG$ is a closed subgroup so that the quotient $\overline{\bG}:=\bG/\ker\rho$ is compact and faithfully represented on $E$ via $\rho$. The goal is to prove the finiteness of $\overline{\bG}$.

      We consider the two possibilities. 
      \begin{itemize}[wide]
      \item {\bf $\overline{\bG}$ is not torsion.} There is then an element $x\in\bG$ generating a closed group with infinite image in $\overline{\bG}$. Restricting $\rho$ to that (abelian!) group gives the desired contradiction.
        
      \item {\bf $\overline{\bG}$ is torsion.} In that case there is \cite[Theorem 1]{wils_tors} a chain
        \begin{equation*}
          \ker\rho=:\bG_0\le \bG_1\le \cdots\le \bG_n:=\bG,
        \end{equation*}
        with
        \begin{equation*}
          \overline{\bG}_i:=\bG_i/\bG_0\trianglelefteq \overline{\bG}\text{ characteristic},
        \end{equation*}
        and with the subquotients
        \begin{equation*}
          \bG_{i+1}/\bG_i\cong \overline{\bG}_{i+1}/\overline{\bG}_i\text{ either }
          \begin{cases}
            \text{pro-$p$ (for various primes $p$)}\\
            \text{or }\cong \bS_i^{\beth_i},\quad \bS_i\text{ finite simple},\quad \beth_i\text{ some cardinal}.
          \end{cases}
        \end{equation*}
        We can now expand the scope of the target claim by substituting for $\ker\rho$ closed normal subgroups to which $\rho$ restricts to a finite-image representation, and thus proceed by induction on $n$, one layer $\bG_i$ at a time. All in all, the goal is to argue that $\rho$ has finite image provided it does so when restricted to some closed $\bH\trianglelefteq \bG$, as well as the pro-$p$ subgroups of $\bG$ for various $p$, and the quotient $\bG/\bH$ is either pro-$p$ or $\bS^{\beth}$ for simple finite $\bS$.

        Now, if $\bG/\bH$ is pro-$p$ then some {\it $p$-Sylow subgroup} \cite[Definition 2.2.1]{wils_prof_bk} $\bK\le \bG$ surjects onto it by general profinite Sylow theory as covered in \cite[\S 2]{wils_prof_bk}, because the {\it supernatural number} \cite[\S 2.1]{wils_prof_bk} $|\bG/\bH|$ divides the largest, possibly infinite $p$ power dividing $|\bG|$, etc. We are now done: $\rho$ is finite when restricted to both $\bH$ and the subgroup $\bK$ normalizing it.

        Assume next that
        \begin{equation*}
          \bG/\bH\cong \bS^{\beth}\cong \prod_{\daleth<\beth}\bS_{\daleth},\quad \bS_{\daleth}\cong \bS
        \end{equation*}
        for some finite simple $\bS$ and cardinal $\beth$. A simple pigeonhole argument shows that the infinitude of $\im\rho$ entails that of the group generated by $\rho(x_{\scriptscriptstyle\daleth})$, $\daleth<\beth$ where $x_{\scriptscriptstyle\daleth}\in \bS_{\daleth}$ is the copy of a single prime-order element $x\in \bS$ in the $\daleth$-indexed copy of $\bS^{\beth}$. But then we can fall back on the preceding branch on the argument, given that $(x_{\scriptscriptstyle\daleth})_{\daleth<\beth}$ belongs to a Sylow subgroup of $\bS^{\beth}$.
      \end{itemize}

      % % The preceding step allows us to assume that $\bG$ is profinite to begin with. The claim, then, is this: if the image through $\rho$ of every abelian subgroup of $\bG$ is finite, then so is $\rho(\bG)$; equivalently, the quotient $\bG/\ker\rho$ is finite (since $\ker\rho\trianglelefteq \bG$ is normal and closed in any case, regardless of the {\it uniform} continuity of $\rho$). Stated thus, the conclusion follows immediately from the fact \cite[Theorem 2]{zel_cpct} that infinite compact groups (e.g. $\bG/\ker\rho$, if assumed infinite for contradiction's sake) have infinite abelian subgroups.
      % % 

      {\bf \Cref{item:clscac-singletor} $\xLeftrightarrow{}$ \Cref{item:clscac-singletor-abim}:} The interesting implication ($\xLeftarrow{}$) follows from the fact \cite[Theorem 2]{zel_cpct} that infinite compact groups have infinite abelian subgroups. 
    \end{enumerate}
  This completes the proof of the theorem.
\end{proof}

%\newpage % THM: END

\begin{remarks}\label{res:hmex}
  \begin{enumerate}[(1), leftmargin=*, wide=0pt]
  \item Cf. also \cite[Exercise E4.8]{hm4}, which parts of \Cref{th:clscpct} echo: there, a representation of a compact group $\bG$ on direct-product (locally convex) space $\bR^I$ is shown to have finitely many isotypic components provided all vectors of the underlying space are locally finite (or $\bG$-finite: \Cref{def:alphaisot}\Cref{item:def:alphaisot-gfin}). In other words, the implication \Cref{item:clscac-fin} $\xLeftarrow{}$ \Cref{item:aprime} of \Cref{th:clscpct} in that slightly different setup.

  \item A version of the equivalence \Cref{item:clscac-fin} $\xLeftrightarrow{}$ \Cref{item:clscac-unif} for {\it unitary} representations of {\it connected} locally compact groups appears as the main result of \cite{zbMATH03289866}.
  \end{enumerate}
\end{remarks}

Given that a compact group is generated topologically by its closed abelian (hence also pro-nilpotent) subgroups, \Cref{th:clscpct}\Cref{item:clscac-maxab} might suggest a kind of local-to-global principle whereby uniformity lifts from a topologically-generating family to the full group $\bG$. No such principle can hold in this generality: every representation is uniform when restricted to the individual factors $\bH_i$ of a product $\bG\cong\prod_i\bH_i$ of {\it finite} groups, without this entailing uniformity on $\bG$. Even for {\it finite} generating families, this cannot go through:

\begin{example}\label{ex:genunif}
  Consider the coproduct $\bG:=\bZ/2\coprod_{\cat{CGp}} \bZ/2$ in the category of compact groups: this is the {\it Bohr compactification} (the {\it AP-compactification} of \cite[\S III.9]{bjm}) of the usual discrete-group coproduct (or free product \cite[Definition preceding Lemma 11.49]{rot-gp})
  \begin{equation*}
    \braket{a}\coprod_{\cat{Gp}}\braket b
    \ \cong\ 
    \bZ/2\coprod_{\cat{Gp}}\bZ/2
    \xrightarrow[\quad\cong\quad]{\text{\cite[Exercise 11.63]{rot-gp}}}
    \bZ\rtimes \bZ/2
    \ \cong\ 
    \braket{ab}\rtimes\braket{a},
  \end{equation*}
  for respective generators $a$ and $b$ for the two copies of $\bZ/2$. Naturally, every unitary representation of $\bG$, including non-uniform ones, will restrict to finite-spectrum $\bZ/2$-representations. The two copies of $\bZ/2$, though, topologically generate $\bG$.
\end{example}

There are also the expected links to the representations induced by \Cref{eq:gonrho} on various algebras of operators on $E$. 

\begin{corollary}\label{cor:actonalgs}
  For a representation \Cref{eq:gonrho} of a compact group on a Banach space the following conditions are equivalent.
  \begin{enumerate}[(a)]

  \item\label{item:cor:actonalgs-rhounif} $\rho$ is uniform.

  \item\label{item:cor:actonalgs-endunif} The conjugation action
    \begin{equation}\label{eq:conjact}
      \bG\ni g\xmapsto{\quad}\rho(g)\cdot\rho(g)^{-1}\in GL\left(\cL(E)\right)
    \end{equation}
    is uniform.

  \item\label{item:cor:actonalgs-idealunif} The conjugation action \Cref{eq:conjact} is uniform when regarded as a representation on any norm-closed ideal $\cI\trianglelefteq \cL(E)$.

  \item\label{item:cor:actonalgs-finidealunif} The conjugation action \Cref{eq:conjact} restricts to a uniform representation on the norm closure
    \begin{equation*}
      \cL_0(E)\cong E\otimes_{\varepsilon}E^*\ (\text{minimal Banach tensor product: \cite[preceding (3.4)]{ryan_ban}})
    \end{equation*}
    of the ideal of finite-rank operators.   
  \end{enumerate}
\end{corollary}
\begin{proof}
  We take for granted the equivalence of uniformity and isotypic finiteness (\Cref{th:clscpct}, \Cref{item:clscac-fin} $\xLeftrightarrow{}$ \Cref{item:clscac-unif}).

  \begin{enumerate}[label={}, leftmargin=*, wide=0pt]

  \item {\bf \Cref{item:cor:actonalgs-rhounif} $\xRightarrow{}$ \Cref{item:cor:actonalgs-endunif}:} Denoting by $E^{\alpha}\le E$ the $\alpha$-isotypic component, the (finite!) linear and topological decomposition
    \begin{equation*}
      E\cong \bigoplus_{\alpha}E^{\alpha}
    \end{equation*}
    gives the corresponding finite decomposition
    \begin{equation*}
      \cL(E)\cong \bigoplus_{\alpha,\beta}\cL\left(E^{\alpha}, E^{\beta}\right). 
    \end{equation*}
    Each summand $\cL\left(E^{\alpha}, E^{\beta}\right)$, as a $\bG$-representation, is of the form $\beta\otimes \alpha^*\otimes F$ for a Banach space $F$, with $\bG$ acting on the $\beta\otimes \alpha^*$ tensorand. The conclusion is now obvious.

  \item {\bf \Cref{item:cor:actonalgs-endunif} $\xRightarrow{}$ \Cref{item:cor:actonalgs-idealunif} $\xRightarrow{}$ \Cref{item:cor:actonalgs-finidealunif}:} all formal.

  \item {\bf \Cref{item:cor:actonalgs-finidealunif} $\xRightarrow{}$ \Cref{item:cor:actonalgs-rhounif}:} For a {\it finite-rank} operator $T\in \cL(E)$ the continuity of
    \begin{equation*}
      \bG\ni g\xmapsto{\quad}\rho(g)T\rho(g)^{-1}\in \cL(E)\quad (\text{norm topology})
    \end{equation*}
    is immediate, and hence so is that of the analogous maps for $T\in \cL_0(E)$. For that reason, the conjugation action \Cref{eq:conjact} does, first off, induce an action on  the Banach space $\cL_0(E)$. But then \Cref{th:clscpct} takes over to show that that action is {\it uniformly} continuous precisely if it has finitely many isotypic components.

    Now, for non-zero isotypic components $E^{\alpha_i}\le E$ the representation $\alpha_j\otimes \alpha_i^*$ appears as a constituent in $\cL_0(E)$. Infinitely many non-zero $E^{\alpha_i}$ will produce infinitely many non-isomorphic summands in the various $\alpha_j\otimes \alpha_i^*$, so that indeed the finite-isotypic-component requirement for $\cL_0(E)$ entails that for $E$.
  \end{enumerate}
\end{proof}

A subtle issue in the proof of \Cref{cor:actonalgs} can be expanded constructively. In the implication \Cref{item:cor:actonalgs-finidealunif} $\xRightarrow{}$ \Cref{item:cor:actonalgs-rhounif}, we had to first argue that the conjugation action actually {\it is} a Banach-space action in the usual sense on the ideal $\cL_0(E)$ of {\it approximable} \cite[p.46, paragraph following (3.3)]{ryan_ban} operators. This is not a moot point: \cite[discussion preceding 2.7.1 Definition]{zbMATH04028180}, for instance, observes that given a compact-group action
\begin{equation*}
  \bG\times A\xrightarrow{\quad\triangleright \quad}A
\end{equation*}
on a $C^*$-algebra $A$, the resulting action on the multiplier algebra $M(A)$ is ``unlikely to be continuous'' for the latter's norm action, in the sense that
\begin{equation*}
  \bG\ni g\xmapsto{\quad}g\triangleright a\in M(A)
\end{equation*}
need not, presumably, be continuous for the {\it norm} topology on $M(A)$ for {\it arbitrary} $a\in M(A)$. This of course connects back to \Cref{cor:actonalgs}: if the representation \Cref{eq:gonrho} is unitary on a Hilbert space $H$, then
\begin{equation*}
  \cL_0(H) = \cK(H)\text{ \cite[\S I.8.1.5]{blk}}
  \quad\text{and}\quad
  M(\cK(H))=\cL(H)\text{ \cite[Example II.7.3.12(ii)]{blk}}. 
\end{equation*}
It turns out that in this unitary-representation setting the aforementioned remark in \cite[\S 2.7]{zbMATH04028180} precisely delineates the uniform representations. We relegate the more general, quantum version to \Cref{th:uniffinspec}, and record the classical consequence here: 

\begin{corollary}\label{cor:actonma}
  Let $\bG$ be a compact group and $\bG\xrightarrow{\rho}\bU(H)$ a unitary representation on a Hilbert space.

  $\rho$ is uniform if and only if the conjugation action \Cref{eq:conjact} is a representation on the Banach space $\cL(H)$, in which case it will automatically be uniform.  \qedhere
\end{corollary}

In fact, one need not {\it start} with a unitary representation:

\begin{corollary}\label{cor:actonuhunif}
  A continuous action of a compact group on the $C^*$-algebra $\cL(H)$ is automatically uniformly continuous, and hence has finite spectrum. 
\end{corollary}
\begin{proof}
  Indeed, because the automorphism group of the $C^*$-algebra $\cL(H)$ is (\cite[Corollary 3 to Theorem 1.4.4]{arv}, \cite[Example II.5.5.14]{blk}) the {\it projective unitary group} $\bP\bU(H):=\bU(H)/\bS^1$ (the group $\cP$ of \cite[\S VIII.1]{vrd}), an action is implemented \cite[Theorem 7.5 and discussion preceding it]{vrd} by a {\it projective representation} of $\bG$ on $H$, and hence \cite[Theorem 7.8]{vrd} by a unitary representation on $H$ of a (compact) group $\bE$ fitting into an exact sequence
  \begin{equation*}
    \{1\}
    \xrightarrow{}
    \bS^1
    \lhook\joinrel\xrightarrow{\quad}    
    \bE
    \xrightarrowdbl{\quad}
    \bG
    \xrightarrow{}
    \{1\}.
  \end{equation*}
  Now apply \Cref{cor:actonma}.
\end{proof}

\subsection{Returning to quantum groups}\label{subse:cqg}

Much of the classical discussion above does not have obvious quantum counterparts: connectedness (\cite[p.3315]{zbMATH05561674}, \cite[Deﬁnition 3.1]{zbMATH06349096}) and identity components \cite[Deﬁnition 3.11]{zbMATH06349096} can be made sense of, but are good deal more problematically (than classically). The usual \cite[Theorem 1.34]{hm4} relationship between total disconnectedness and profiniteness fails massively \cite[Definition 3.21 and Proposition 3.23]{zbMATH06349096}, the {\it supplement theorem} of Dong Hoon Lee \cite[Theorem 9.41]{hm4}, used in the proof of \Cref{th:clscpct} (implication \Cref{item:clscac-g0profin} $\xRightarrow{}$ \Cref{item:clscac-clim}) is absent, etc. The crucial link between uniformity and spectrum compactness does survive in at least some of its guises though.

% % , in its most basic unitary-representation guise, captured by the equivalence \Cref{item:clscac-fin} $\xLeftrightarrow{}$ \Cref{item:clscac-g2gl} of \Cref{th:clscpct}. 
% %

\begin{theorem}\label{th:uniffinspec}
  For a unitary representation \Cref{eq:uonh} of a compact quantum group $\bG$ the following conditions are equivalent.
  \begin{enumerate}[(a)]
  \item\label{item:th:uniffinspec-finspec} $U$ has finite spectrum.

  \item\label{item:th:uniffinspec-unifactonbh} The conjugation $\bG$-action
    \begin{equation}\label{eq:actonkh}
      \cK(H)\ni T
      \xmapsto{\quad}
      U^*(1\otimes T)U
      \in \cC(\bG)\underline{\otimes}\cK(H)
    \end{equation}
    lifts to a uniform action on all of $\cL(H)$.

  \item\label{item:th:uniffinspec-actonbh} The action \Cref{eq:actonkh} lifts to an action on $\cL(H)$ (not assumed uniform, a priori).
  \end{enumerate}
\end{theorem}
\begin{proof}
  That \Cref{item:th:uniffinspec-finspec} $\xRightarrow{}$ \Cref{item:th:uniffinspec-unifactonbh} follows from \Cref{re:uisintens} and \Cref{item:th:uniffinspec-unifactonbh} $\xRightarrow{}$ \Cref{item:th:uniffinspec-actonbh} is formal. We address the implication \Cref{item:th:uniffinspec-actonbh} $\xRightarrow{}$ \Cref{item:th:uniffinspec-finspec}, assuming
  \begin{equation}\label{eq:conjactlh}
    \cL(H)
    \ni
    T
    \xmapsto{\quad}
    U^*(1\otimes T)U
    \in
    \cC(\bG)\underline{\otimes}\cL(H)
  \end{equation}
  to be a $\bG$-action on the $C^*$-algebra $\cL(H)$ in the customary sense \cite[Definition 1.4]{podl_symm}.

  Suppose the spectrum of the original representation contains infinitely many $\alpha_i\in \Irr(\bG)$, whence also some
  \begin{equation*}
    \beta_n:=\alpha_{k_n}
    ,\quad
    \beta_{n+1}\otimes\beta_n^*\ \text{mutually disjoint (i.e. no common irreducible summands)}.
  \end{equation*}
  But then norm-1 finite-rank operators
  \begin{equation*}
    H_n:=H^{\beta_n}\xrightarrow{\quad T_n\quad} H^{\beta_{n+1}}=:H_{n+1}
  \end{equation*}
  glue into a single norm-1 operator
  \begin{equation*}
    T\in \cL(H),\quad T|_{H^{\beta_n}} = T_n,
  \end{equation*}
  which by construction has distance $\ge 1$ from any $\bG$-finite operator under the action \Cref{eq:conjactlh}: indeed, for such a $\bG$-finite element $S\in\cL(H)$, by our very choice of $(\beta_n)_n$,
  \begin{equation*}
    \left(\text{projection onto }H^{\beta_{n+1}}\right)
    \circ
    S
    \circ
    \left(\text{projection onto }H^{\beta_{n}}\right)
    =0
  \end{equation*}
  for all but finitely many $n$. But then, because the $\bG$-finite vectors are {\it not} norm-dense in $\cL(H)$, \Cref{eq:conjactlh} cannot \cite[Corollary 1.6]{podl_symm} be a $\bG$-action in the sense of \cite[Definition 1.4]{podl_symm}. 
\end{proof}

%\newpage

%%%%%%%%%%%%%%%%%%%%%%%%%%%%%%%%
%%%%%%%%%%%%%%%%%%%%%%%%%%%%%%%%
\section{Discreteness as action-coproduct preservation}\label{se:discrcopr}

The discussion below will make repeated references to the {\it Stone-\v{C}ech compactification} functor
\begin{equation}\label{eq:beta}
  \cat{Top}
  \xrightarrow{\quad\beta\quad}
  \cat{Top}_{CH}
\end{equation}
from topological to compact Hausdorff spaces, i.e. (\cite[\S V.8]{mcl_2e}, \cite[Definition preceding Lemma 2.4]{cn_ultrafilt}) the left adjoint to the inclusion functor in the opposite direction. 

\begin{theorem}\label{th:discr-coprod-pres}
  For a locally compact group $\bG$ the following conditions are equivalent.

  \begin{enumerate}[(a)]

  \item\label{item:th:discr-coprod-pres-isdiscr} $\bG$ is discrete.

  \item\label{item:th:discr-coprod-pres-sc} Continuous $\bG$-actions lift along Stone-\v{C}ech compactifications $X\xrightarrow{}\beta X$.

  \item\label{item:th:discr-coprod-pres-mult} Continuous $\bG$-actions on $C^*$-algebras $A$ lift along embeddings $A\lhook\joinrel\xrightarrow{}M(A)$ into multiplier algebras. 

  \item\label{item:th:discr-coprod-pres-prod} The forgetful functor $\tensor*[]{\cC}{^*^\bG_1}\xrightarrow{}\cC^*_1$ from unital $\cC^*$-algebras equipped with a $\bG$-action to unital $\cC^*$-algebras preserves products. 
    
  \item\label{item:th:discr-coprod-pres-coprodact} The forgetful functor $\cat{Top}_{CH}^{\bG}\xrightarrow{}\cat{Top}_{CH}$ from compact Hausdorff spaces carrying $\bG$-actions to compact Hausdorff spaces preserves coproducts.

  \end{enumerate}
\end{theorem}
\begin{proof}    
  \begin{enumerate}[label={}, leftmargin=*, wide=0pt]
  \item {\bf \Cref{item:th:discr-coprod-pres-isdiscr} $\xRightarrow{}$ \Cref{item:th:discr-coprod-pres-sc}, \Cref{item:th:discr-coprod-pres-mult}, \Cref{item:th:discr-coprod-pres-prod}, \Cref{item:th:discr-coprod-pres-coprodact}:} In all cases, the continuity requirement on the various actions amounts to the continuity of a map
    \begin{equation*}
      \bG\xrightarrow{\quad}\cC(c,c) = \End_{\cC}(c)
    \end{equation*}
    for some object $c$ in a category {\it enriched} \cite[\S 2.2]{zbMATH01512858} over $\cat{Top}$. Naturally, if $\bG$ is discrete any such condition will be moot.
    
    % % 
    % % Continuity of a $\bG$-action on a space $X$ means continuity for the corresponding map from $\bG$ to an appropriately-topologized space of self-maps on $X$, so such a condition of course holds automatically when $\bG$ is discrete.
    % % 
    
  \item {\bf \Cref{item:th:discr-coprod-pres-sc} $\xRightarrow{}$ \Cref{item:th:discr-coprod-pres-coprodact}:} Coproducts in $\cat{Top}_{CH}$ are formed by applying the Stone-\v{C}ech compactification functor \Cref{eq:beta} to the disjoint union of topological spaces, so condition \Cref{item:th:discr-coprod-pres-coprodact} amounts to the requirement that $\bG$-actions on $X_i$, $i\in I$ glue to a continuous action on $\beta\left(\coprod_i X_i\right)$. Similarly:
    
  \item {\bf \Cref{item:th:discr-coprod-pres-mult} $\xRightarrow{}$ \Cref{item:th:discr-coprod-pres-prod}:} The product in the category $\cC^*_1$ of unital $C^*$-algebras $A_i$, $i\in I$ is
    \begin{equation*}
      \ell_{\infty}\left((A_i)_i\right)
      :=
      \left\{(a_i)_i\in \prod_i A_i\ |\ \|a_i\|\text{ bounded}\right\}
    \end{equation*}
    (as noted also in \cite[\S 2.1]{ped-pp}), and hence the multiplier algebra of
    \begin{equation}\label{eq:c0cast}
      c_0\left((A_i)_i\right)
      :=
      \left\{(a_i)_i\in \prod_i A_i\ |\ \|a_i\|\text{ vanishing at infinity}\right\}.
    \end{equation}
    Continuous $\bG$-actions on the $A_i$ are easily seen to induce one on \Cref{eq:c0cast}, so that action's liftability to the multiplier algebra yields \Cref{item:th:discr-coprod-pres-prod}. 

  \item {\bf \Cref{item:th:discr-coprod-pres-prod} $\xRightarrow{}$ \Cref{item:th:discr-coprod-pres-coprodact}:} The latter is nothing but the former's specialization to {\it commutative} unital $C^*$-algebras. 
    
  \item {\bf \Cref{item:th:discr-coprod-pres-coprodact} $\xRightarrow{}$ \Cref{item:th:discr-coprod-pres-isdiscr}:} $\bG$ acts on itself by (left, say) translation, so there is \cite[Proposition 3.1]{dvr-ex} a {\it $\bG$-equivariant compactification} $b\bG\supseteq \bG$. We will apply the hypothesis to $\Lambda$-indexed copies of the resulting action
    \begin{equation*}
      \bG\times b\bG
      \xrightarrow{\quad}
      b\bG.
    \end{equation*}
    Denoting $\bG_{\lambda}:=\bG$, the goal is thus to show that the action
    \begin{equation*}
      \bG\times
      \beta\left(\coprod_{\lambda\in\Lambda}b\bG_{\lambda}\right)
      \xrightarrow{\quad}
      \beta\left(\coprod_{\lambda\in\Lambda}b\bG_{\lambda}\right)
      =:X
    \end{equation*}
    is discontinuous.

    Suppose $\bG$ is {\it not} discrete. We can then find a convergent {\it net} \cite[Definition 11.2]{wil-top}
    \begin{equation}\label{eq:gonbgx}
      1\ne g_{\lambda}
      \xrightarrow[\lambda]{\quad}
      1
      ,\quad
      \lambda\in (\Lambda,\le).
    \end{equation}
    Furthermore, we can select the $g_{\lambda}$ so that compactly-supported continuous functions $f_{\lambda}$ exist with
    \begin{equation*}
      f_{\lambda}(1)=1,\quad f_{\lambda}(g_{\lambda})=0,\quad\forall \lambda\in\Lambda. 
    \end{equation*}
    Because the continuous functions on $\bG$ obtained by restriction through $\bG\subseteq b\bG$ are \cite[Example 2.3]{dvr-puc} precisely the (bounded) {\it right-uniformly continuous} (i.e. those that are uniformly continuous with respect to the {\it right uniformity} of \cite[\S 7.1, p.106]{zbMATH01285882}). In particular, the $f_{\lambda}$ (being compactly-supported) can be regarded as continuous functions on $b\bG$.
    
    % % To do {\it that}, one need not bother with the inner `$b$' decorations: given that the compactification $\bG\subseteq b\bG$ is a homeomorphism onto its image, it suffices to argue that
    % % \begin{equation}\label{eq:gonx}
    % %   \bG\times
    % %   \beta\left(\coprod_{\lambda\in\Lambda}\bG_{\lambda}\right)
    % %   \xrightarrow{\quad}
    % %   \beta\left(\coprod_{\lambda\in\Lambda}\bG_{\lambda}\right)
    % %   =X
    % % \end{equation}
    % % is discontinuous instead. 
    % % 
    
    Passing to a subnet if necessary, we can assume \cite[Theorem 17.4]{wil-top}
    \begin{equation*}
      (1_{\lambda})_{\lambda},\quad 1_{\lambda}:=1\in G_{\lambda}
    \end{equation*}
    convergent to some $p\in X$. Continuity for \Cref{eq:gonbgx} would demand that
    \begin{equation*}
      g_{\lambda} = g_{\lambda}\cdot 1_{\lambda}
      \xrightarrow[\lambda]{\quad}
      1\cdot p=p\in X,
    \end{equation*}
    which will not be the case under the assumption that $\bG$ is not discrete: we can find \cite[Theorem 15.8]{wil-top} continuous functions
    \begin{equation*}
      \bG\xrightarrow{\quad f_{\lambda}\quad}[0,1]
      ,\quad
      f(1)=1,\ f(g_{\lambda})=0,
    \end{equation*}
    and these glue to a continuous function $X\xrightarrow{f}[0,1]$ with $f(p)=1$ and $g_{\lambda}=0$ for all $\lambda$.
  \end{enumerate}
  This concludes the proof. 
\end{proof}

As for recovering some quantum flavor of the above, (co)products will be less suitable because $\cC(\bG)$'s failure to be commutative for quantum $\bG$ makes gluing $\bG$-action $A_i\xrightarrow{} \cC(\bG)\underline{\otimes}A_i$ into actions on $\prod_i A_i$ a tall order. Lifting actions to multiplier algebras, as in
\begin{equation*}
  \text{
    \Cref{item:th:discr-coprod-pres-isdiscr}
    $\xLeftrightarrow{\quad}$
    \Cref{item:th:discr-coprod-pres-sc}
    $\xLeftrightarrow{\quad}$
    \Cref{item:th:discr-coprod-pres-mult}
  }
\end{equation*}
of \Cref{th:discr-coprod-pres}, {\it does} make sense however. Recall \cite[Deﬁnition 2.6]{vs-impr} that for a compact quantum group $\bG$ and a possibly non-unital $C^*$-algebra $A$ a {\it (continuous) $\bG$-action on $A$} is a {\it non-degenerate} \cite[p.840]{kvcast} morphism
\begin{equation}\label{eq:nonunitact}
  A\xrightarrow{\quad\rho\quad}M(\cC(\bG)\underline{\otimes}A),
\end{equation}
coassociative in the usual sense that
\begin{equation*}
  \begin{tikzpicture}[auto,baseline=(current  bounding  box.center)]
    \path[anchor=base] 
    (0,0) node (l) {$A$}
    +(3,.5) node (u) {$M(\cC(\bG)\underline{\otimes}A)$}
    +(3,-.5) node (d) {$M(\cC(\bG)\underline{\otimes}A)$}
    +(7,0) node (r) {$M\left(\cC(\bG)^{\underline{\otimes}2}\underline{\otimes}A\right)$}
    ;
    \draw[->] (l) to[bend left=6] node[pos=.5,auto] {$\scriptstyle \rho$} (u);
    \draw[->] (u) to[bend left=6] node[pos=.5,auto] {$\scriptstyle \id\otimes\rho$} (r);
    \draw[->] (l) to[bend right=6] node[pos=.5,auto,swap] {$\scriptstyle \rho$} (d);
    \draw[->] (d) to[bend right=6] node[pos=.5,auto,swap] {$\scriptstyle \Delta\otimes\id$} (r);
  \end{tikzpicture}
\end{equation*}
commutes, and such that
\begin{equation}\label{eq:actclosedspan}
  \overline{\spn}^{\|\cdot\|}\left(\cC(\bG)\underline{\otimes}\bC\right)\rho(A) = M(\cC(\bG)\underline{\otimes}A).
\end{equation}

\begin{remark}
  For \Cref{eq:actclosedspan} to make sense, it is understood (\cite[Definition 1.8]{zbMATH04119008}, \cite[\S 3.1]{vrgn-phd}) that $\rho$ in fact takes values in the subalgebra
  \begin{equation*}    
    \widetilde{M}(\cC(\bG)\underline{\otimes}A)
    \le
    M(\cC(\bG)\underline{\otimes}A)
  \end{equation*}
  consisting of those elements which multiply $\cC(\bG)\underline{\otimes}\bC$ into $\cC(\bG)\underline{\otimes}A$. This is an additional restriction on multipliers: $A$ being non-unital, $\cC(\bG)\underline{\otimes}\bC$ will of course not be a subspace of $\cC(\bG)\underline{\otimes}A$. 
\end{remark}

\begin{theorem}\label{th:cqgisfin}
  For a compact quantum group $\bG$ the following conditions are equivalent.
  \begin{enumerate}[(a)]
  \item\label{item:th:cqgisfin-fin} $\bG$ is finite, i.e. $\cC(\bG)=\cO(\bG)$ is finite-dimensional.

  \item\label{item:th:cqgisfin-ma} Every action \Cref{eq:nonunitact} extends to one on $M(A)$.

  \item\label{item:th:cqgisfin-aaast} Every action \Cref{eq:nonunitact} extends to one on the {\it von Neumann envelope} \cite[Chapter III, Theorem 2.4 and Definition 2.5]{tak1} $A^{**}\supseteq A$. 
  \end{enumerate}  
\end{theorem}
\begin{proof}
  If $\bG$ is finite all tensor products are purely algebraic, so that \Cref{item:th:cqgisfin-fin} does indeed imply both \Cref{item:th:cqgisfin-ma} and \Cref{item:th:cqgisfin-aaast}. Conversely, for every Hilbert space $H$ we have \cite[Th\'eor\`eme 2]{zbMATH03055671} an identification
  \begin{equation*}
    \begin{tikzpicture}[auto,baseline=(current  bounding  box.center)]
      \path[anchor=base] 
      (0,0) node (l) {$\cK(H)$}
      +(2,.5) node (u) {$\cK(H)^{**}$}
      +(4,0) node (r) {$\cL(H)$}
      ;
      \draw[right hook->] (l) to[bend left=6] node[pos=.5,auto] {$\scriptstyle $} (u);
      \draw[->] (u) to[bend left=6] node[pos=.5,auto] {$\scriptstyle \cong$} (r);
      \draw[right hook->] (l) to[bend right=6] node[pos=.5,auto,swap] {$\scriptstyle $} (r);
    \end{tikzpicture}
  \end{equation*}
  of the two obvious inclusions, hence the opposite implications \Cref{item:th:cqgisfin-ma}, \Cref{item:th:cqgisfin-aaast} $\xRightarrow{}$ \Cref{item:th:cqgisfin-fin}: simply apply \Cref{th:uniffinspec} to conclude that all unitary $\bG$-representations have finite spectrum. 
\end{proof}

%\newpage

%%%%%%%%%%%%%%%%%%%%%%%%%%%%%%%%
%%%%%%%%%%%%%%%%%%%%%%%%%%%%%%%%
\section{Jointly monic families and uniformity}\label{se:topgen}

Recall \cite[Definition 4]{bcv_connes} that a family
\begin{equation}\label{eq:g2hi}
  \cO(\bG)\xrightarrowdbl{\quad}\cO(\bH_i),\quad i\in I
\end{equation}
of quantum subgroups {\it topologically generates} $\bG$ if said quotients do not all factor through a proper quotient $\cO(\bG)\xrightarrowdbl{}\cO(\bH)$. Equivalently, the topologically-generating families of quotients \Cref{eq:g2hi} are precisely those whose induced corestriction functors between the respective categories of comodules are {\it jointly full} \cite[Definition 2.15]{chi_rf}: a linear map between two (right-say) $\cO(\bG)$-comodules is a comodule morphism if and only if it is a comodule morphism over every quotient $\cO(\bH_i)$ in \Cref{eq:g2hi}.

Given the equivalence \cite[Theorem 2.1]{agore_mono}, for a coalgebra (hence also bialgebra or Hopf algebra) morphism $C\xrightarrow{f} D$, between $f$ being {\it monic} \cite[\S I.5]{mcl_2e} in the respective category and its inducing a full corestriction functor between categories of comodules, the following notion \cite[Definitiopn 10.5]{ahs} will be of some relevance. 

\begin{definition}\label{def:jointmon}
  A family of morphisms $c\xrightarrow{f_i}c_i$ in a category $\cC$ is {\it jointly monic} or a {\it mono-source} if for every $d\in \cC$ the map
  \begin{equation*}
    \cC(d,c)
    \xrightarrow{\quad \left(f_i\circ\right)_i\quad}
    \prod_i \cC(d,c_i)
  \end{equation*}
  is injective.

  {\it Jointly epic} families (or {\it epi-sinks} \cite[\S 10.63]{ahs}) are defined dually. 
\end{definition}

The family/source versions of the single-morphism results for coalgebras \cite[Theorem 2.1]{agore_mono} or Hopf algebras \cite[Theorem 2.5]{2302.12870v1} or the obvious analogues for bialgebras and so on are now no more difficult to prove, so we provide only the statement(s). 

\begin{proposition}\label{pr:monosourcecomod}
  A family of coalgebra morphisms $C\to C_i$ over a field $\Bbbk$ is a mono-source if and only if the corresponding family of corestriction functors $\cM^{C}\xrightarrow{}\cM^{C_i}$ between categories of comodules is jointly full.

  Moreover, the analogous statement holds in the category of Hopf $\Bbbk$-algebras, $\Bbbk$-bialgebras and CQG algebras, and one can substitute categories of finite-dimensional comodules throughout.  \qedhere
\end{proposition}

With this in place, topologically-generating families $\bH_i\le \bG$ of compact quantum subgroups are nothing but epi-sinks in the category of compact quantum groups, dual to the category $\cat{CQGAlg}$ of CQG algebras. Or: they are families dual to mono-sources $\cO(\bG)\xrightarrowdbl{}\cO(\bH_i)$ in $\cat{CQGAlg}$. 

We noted in \Cref{ex:genunif} that the finite-spectrum condition does not lift from the individual members of a finite topologically generating family. One crucial aspect of the proof of the implication \Cref{item:clscac-g0profin} $\xRightarrow{}$ \Cref{item:clscac-clim} of \Cref{th:clscpct}, though, is that $\bG$ was generated by $\bD$ and the {\it normal} subgroup $\bG_0\trianglelefteq \bG$.

We transfer that setup here, while also extending the scope of the discussion to {\it linearly reductive} quantum groups: objects dual to cosemisimple Hopf algebras. We retain the notation $\cO(\bG)$ for the object (Hopf algebra over a field $\Bbbk$) dual to a quantum group $\bG$. Quantum subgroups $\bH\le \bG$ are then (dual to) Hopf algebra quotients $\cO(\bG)\xrightarrowdbl{}\cO(\bH)$, etc.

A quantum subgroup $\bH\le \bG$ is {\it normal} \cite[\S 1.5, pp.10-11]{pw} if the corresponding quotient $\cO(\bG)\xrightarrowdbl{}\cO(\bH)$ is invariant under both the left and the right {\it adjoint coactions}:
\begin{equation*}
  \cO(\bG)\ni x
  \xmapsto{\quad}
  x_1 S(x_3)\otimes x_2
  \quad\text{or}\quad
  x_2\otimes S(x_1)x_3
  \in
  \cO(\bG)^{\otimes 2}
\end{equation*}
respectively, in {\it Sweedler notation} \cite[\S 1.2]{swe}. The two conditions will be equivalent \cite[Proposition 1.5.1]{pw} for surjections of Hopf algebras with bijective antipode.

\begin{theorem}\label{th:fintopgen}
  Let $\bK,\bH\le \bG$ be an epi-sink of linearly reductive quantum groups over a field $\Bbbk$, with $\bK$ normal.

  A $\bG$-comodule $V\xrightarrow{\rho}V\otimes\cO(\bG)$ has finitely many irreducible constituents over $\bH$ and $\bK$ if and only if it does so over $\bG$. 
\end{theorem}
\begin{proof}
  Every simple $\cO(\bG)$-comodule is finite-dimensional \cite[Finiteness Theorem 5.1.1]{mont} and hence decomposes into finitely many summands as either a comodule over any cosemisimple quotient of $\cO(\bG)$. The interesting implication is thus ($\xRightarrow{}$).

  All comodules over a cosemisimple coalgebra being {\it coflat} \cite[\S 2.4, preceding Theorem 2.4.17]{dnr}, the normality of $\bK\le \bG$ is in fact equivalent \cite[Theorem 2]{tak-ff} to the surjection $\cO(\bG)\xrightarrowdbl{}\cO(\bK)$ fitting into an {\it exact sequence} \cite[p.23]{ad}
  \begin{equation*}
    \Bbbk
    \xrightarrow{}
    \cO(\bQ)
    \lhook\joinrel\xrightarrow{\quad}    
    \cO(\bG)
    \xrightarrowdbl{\quad}
    \cO(\bK)
    \xrightarrow{}
    \Bbbk
  \end{equation*}
  of (cosemisimple) Hopf algebras. 
  
  The assumption that $V$ have finitely many irreducible constituents as a $\cO(\bK)$-comodule means \cite[Theorem 4.4]{zbMATH07700069} that its irreducible constituents $V_i$, $i\in I$ over the original (large) Hopf algebra $\cO(\bG)$ fall into finitely many classes for the equivalence relation
  \begin{equation}\label{eq:vivj}
    V_i\sim V_j
    \quad
    \xLeftrightarrow{\quad}
    \quad
    V_i\le V_j\otimes \cO(\bQ)
    \text{ as $\cO(\bG)$-comodules}.
  \end{equation}
  On the other hand the fact that the pair of surjections
  \begin{equation*}
    \cO(\bG)\xrightarrowdbl{\quad}\cO(\bK),\ \cO(\bH)
  \end{equation*}
  is a Hopf-algebra mono-source implies that the composition
  \begin{equation*}
    \cO(\bQ)
    \lhook\joinrel\xrightarrow{\quad}    
    \cO(\bG)
    \xrightarrowdbl{\quad}
    \cO(\bH)
  \end{equation*}
  is an injection, so that whenever $V_i$ is {\it not} equivalent to $V_j$ in the sense of \Cref{eq:vivj} the same holds over $\cO(\bH)$. But this means that each of the (finitely many) conjugacy of \Cref{eq:vivj} is itself finite, and we are done. 
\end{proof}

\addcontentsline{toc}{section}{References}
%\bibliography{bib}{}
%\bibliographystyle{plain}
%\bibliographystyle{emss}

\Addresses

\end{document}